\documentclass[12pt]{amsart}
 \usepackage{amscd,amsmath,amsthm,amssymb}
 \usepackage{pstcol,pst-plot,pst-3d}
 \psset{unit=0.7cm,linewidth=0.8pt,arrowsize=2.5pt 4}

\newpsstyle{fatline}{linewidth=1.5pt}
\newpsstyle{fyp}{fillstyle=solid,fillcolor=verylight}
\definecolor{verylight}{gray}{0.97}
\definecolor{light}{gray}{0.93}
\definecolor{medium}{gray}{0.82}
 \def\frk{\frak}               

 \def\mm{{\frk m}}
 
 \def\nn{{\frk n}}
 \def\MI{{\mathcal I}}

 \def\P{{\mathcal P}}

 \def\opn#1#2{\def#1{\operatorname{#2}}} 
 \opn\chara{char} \opn\length{\ell} \opn\pd{pd} \opn\rk{rk}
 \opn\projdim{proj\,dim} \opn\injdim{inj\,dim} \opn\rank{rank}
 \opn\depth{depth} \opn\grade{grade} \opn\height{height}
 \opn\embdim{emb\,dim} \opn\codim{codim}
 
 \opn\Tr{Tr} \opn\bigrank{big\,rank}
 \opn\superheight{superheight}\opn\lcm{lcm}
 \opn\trdeg{tr\,deg}
 \opn\reg{reg} \opn\lreg{lreg} \opn\ini{in} \opn\lpd{lpd}
 \opn\size{size} \opn\sdepth{sdepth}
 \opn\link{link}\opn\fdepth{fdepth}\opn\lex{lex}
 \opn\div{div} \opn\Div{Div} \opn\cl{cl} \opn\Cl{Cl}
 \opn\Spec{Spec} \opn\Supp{Supp} \opn\supp{supp} \opn\Sing{Sing}
 \opn\Ass{Ass} \opn\Min{Min}\opn\Mon{Mon}
 \opn\Ann{Ann} \opn\Rad{Rad} \opn\Soc{Soc}
 \opn\Im{Im} \opn\Ker{Ker} \opn\Coker{Coker} \opn\Am{Am}
 \opn\Hom{Hom} \opn\Tor{Tor} \opn\Ext{Ext} \opn\End{End}
 \opn\Aut{Aut} \opn\id{id}
 
 \opn\nat{nat}
 \opn\pff{pf}
 \opn\Pf{Pf} \opn\GL{GL} \opn\SL{SL} \opn\mod{mod} \opn\ord{ord}
 \opn\Gin{Gin} \opn\Hilb{Hilb}\opn\sort{sort}
 \opn\aff{aff} \opn
 \con{conv} \opn\relint{relint} \opn\st{st}
 \opn\lk{lk} \opn\cn{cn} \opn\core{core} \opn\vol{vol}
 \opn\link{link} \opn\star{star}\opn\lex{lex}\opn\set{set}
 \opn\gr{gr}
 
 \def\pot#1#2{#1[\kern-0.28ex[#2]\kern-0.28ex]}

 \opn\dirlim{\underrightarrow{\lim}}
 \opn\inivlim{\underleftarrow{\lim}}
 \let\union=\cup
 \let\sect=\cap

 \let\Union=\bigcup
 \let\Sect=\bigcap
 \let\Dirsum=\bigoplus
 
 \let\To=\longrightarrow
 \def\Implies{\ifmmode\Longrightarrow \else
         \unskip${}\Longrightarrow{}$\ignorespaces\fi}
 \def\implies{\ifmmode\Rightarrow \else
         \unskip${}\Rightarrow{}$\ignorespaces\fi}
 \def\iff{\ifmmode\Longleftrightarrow \else
         \unskip${}\Longleftrightarrow{}$\ignorespaces\fi}

 \let\:=\colon
 \newtheorem{Theorem}{Theorem}[section]
 \newtheorem{Lemma}[Theorem]{Lemma}
 \newtheorem{Corollary}[Theorem]{Corollary}
 \newtheorem{Proposition}[Theorem]{Proposition}

 \newtheorem{Example}[Theorem]{Example}
 
 \newtheorem{Definition}[Theorem]{Definition}
 
 \newtheorem{Conjecture}[Theorem]{Conjecture}
 
 \let\epsilon\varepsilon
 \let\kappa=\varkappa
 \def\qed{\ifhmode\textqed\fi
       \ifmmode\ifinner\quad\qedsymbol\else\dispqed\fi\fi}
 \def\textqed{\unskip\nobreak\penalty50
        \hskip2em\hbox{}\nobreak\hfil\qedsymbol
        \parfillskip=0pt \finalhyphendemerits=0}
 \def\dispqed{\rlap{\qquad\qedsymbol}}
 
 \opn\dis{dis}
 \def\pnt{{\raise0.5mm\hbox{\large\bf.}}}
 
 \opn\Lex{Lex}

\begin{document}

 \title {Monomial localizations and polymatroidal ideals}

 \author {Somayeh Bandari and J\"urgen Herzog}

\address{Somayeh Bandari, Department of Mathematics, Az-Zahra
University, Vanak, Post Code 19834, Tehran, Iran}
\email{somayeh.bandari@yahoo.com}

\address{J\"urgen Herzog, Fachbereich Mathematik, Universit\"at Duisburg-Essen, Campus Essen, 45117
Essen, Germany}
\email{juergen.herzog@uni-essen.de}


 \begin{abstract}
In this paper we consider monomial localizations of monomial ideals and conjecture that a monomial ideal is polymatroidal if and only if all its monomial localizations have a linear resolution. The conjecture is proved for squarefree monomial ideals where it is equivalent to a well-known characterization of matroids. We prove our conjecture in many other special cases. We also introduce the concept of componentwise polymatroidal ideals and extend several of the results, known for polymatroidal ideals,  to this new  class of ideals.
 \end{abstract}

 \maketitle

 \section*{Introduction}
The class of polymatroidal ideals is one of the rare classes of monomial ideals with the property that all powers of an ideal in this class have a linear resolution. This is due to the fact that the powers of a polymatroidal ideal are again polymatroidal \cite[Theorem 5.3]{CH} and that polymatroidal ideals have linear quotients \cite[Lemma 1.3]{HT} which implies that they have linear resolutions. Recall that a monomial ideal is called polymatroidal, if its monomial generators correspond to the bases of a discrete polymatroid, see \cite{HH}. Since the set of bases of a discrete polymatroid is characterized by the so-called exchange property, it follows that a polymatroidal ideal may as well be characterized as follows: let $I\subset S=K[x_1,\ldots,x_n]$ be a monomial ideal generated in a single degree. We denote, as usual by $G(I)$ the unique minimal set of monomial generators of $I$. Then $I$ is said to be polymatroidal, if for any two elements $u,v\in G(I)$ such that $\deg_{x_i}(u)> \deg_{x_i}(v)$ there exists an index $j$  with $\deg_{x_j}(u)< \deg_{x_j}(v)$ such that $x_j(u/x_i)\in I$.

Recently it has been observed that a monomial localization of a polymatroidal is again polymatroidal \cite[Corollary 3.2]{HRV}.
The monomial localization  of a monomial ideal $I$ with respect to a monomial prime ideal $P$ is the monomial ideal $I(P)$ which is obtained from $I$ by substituting the variables $x_i\not \in P$ by $1$. Observe that $I(P)$ is the unique monomial ideal with the property that $I(P)S_P=IS_P$. The monomial localization $I(P)$ can also be described as the saturation
$I\: (\prod_{x_i\not\in P}x_i)^\infty$.  Thus in the case that the polymatroidal ideal $I$  is squarefree, in which case it is called matroidal, we see  that $I(P)=I\:u$ where $u=\prod_{x_i\not\in P}x_i$.

By what we have explained so far it follows that all monomial localizations of polymatroidal ideals have a linear resolution. The natural question arises whether this property characterizes  polymatroidal ideals. The main purpose of this paper is to discuss this question. In Theorem~\ref{easier} we give an affirmative answer requiring however more that just the condition that all monomial localizations have a linear resolution. To be precise we show, that a monomial ideal $I$ is polymatroidal if and only $I\: u$ has a linear resolution for all monomials $u$ in $S$. In fact, among a few other equivalent conditions, we also  show that $I$ is polymatroidal if we only require that $I\: u$ is generated in a single degree for all monomials $u\in S$.   Since for a squarefree  monomial ideal $I$,  the colon ideal $I\: u$  is  a monomial localization for any monomial $u$, it follows (see Corollary~\ref{newyear}) that a squarefree monomial ideal $I$ is matroidal if and only if $I(P)$ is generated in a single degree for all monomial prime ideals $P$. It turns out that this characterization of matroidal ideals corresponds to a well-known characterization of matroids which  says that a simplicial complex is a matroid if and only if all its induced subcomplexes are pure, see  \cite[Proposition 3.1]{S}.

Even though  matroidal ideals are characterized by the property that all its monomial localizations have a linear resolution, we don't know whether the corresponding statement is true for polymatroidal ideals. There are simple examples of monomial ideals which show that all monomial localizations are generated in a single degree but the ideals themselves  are not polymatroidal. However due to computational evidence we are lead to conjecture that the  monomial ideals with the property that all monomial localizations have a linear resolution are precisely the polymatroidal ideals. In Section~\ref{polymatroidalsection} we discuss several special cases which support this conjecture. In fact we give an affirmative answer to the conjecture in the following cases: 1.\ $I$  is generated in degree 2 (Proposition~\ref{somayehbelievesit}), 2.\ $I$ contains at least $n-1$ pure powers (Proposition~\ref{purepowers}), 3.\ $I$ is monomial ideal in at most 3 variables (Corollary~\ref{dimtwo} and Proposition~\ref{dimthree}), 4.\  $I$ has no embedded prime ideal and either $|\Ass(S/I)|\leq 3$  or $\height(I)=n-1$ (Proposition~\ref{anothercase}).

We would like to point out that in each of the special cases mentioned above we use completely   different arguments for  the proof of our conjecture.  For the moment we do not have a general strategy to prove it.

In Section~\ref{extension} we introduce componentwise polymatroidal ideals, namely those monomial ideals with the property that each of its components is generated by  a polymatroidal ideal. In contrast to polymatroidal ideals, powers of componentwise polymatroidal ideals need not to be componentwise polymatroidal, unless the  ideal is generated in at most two degrees, see Proposition \ref{veryspecial}. On the other, it might be that powers of componentwise linear ideals are componentwise linear. For this we could not find a counter example.

One would expect that an exchange property of its generators characterizes componentwise polymatroidal ideals.  For that purpose we introduce the so-called non-pure exchange property and show in Proposition \ref{exchange} that componentwise polymatroidal ideals enjoy the non-pure exchange property. On the other hand, we show by an example that an ideal with the non-pure exchange property need not to be componentwise polymatroidal.

It is natural to ask whether componentwise polymatroidal ideals have linear quotients. We expect that this is the case and prove it for  ideals which are componentwise of Veronese type. It is also an open question whether ideals satisfying the non-pure exchange property have linear quotients, even they are not componentwise polymatroidal.

 \section{An algebraic characterization of polymatroidal ideals and monomial localizations of  matroidal ideals}
\label{algebraicsection}
 Let $K$ be a field, $S=K[x_1,\ldots,x_n]$ the polynomial ring in the indeterminates  $x_1,\ldots, x_n$ and $I\subset S$ a monomial ideal. We first show

\begin{Theorem}
\label{easier}
Let $I$ be a monomial ideal. The following conditions are equivalent:
\begin{enumerate}
\item[(a)] $I$ is polymatroidal.
\item[(b)] $I\:u$ is polymatroidal  for all monomials $u$.
\item[(c)] $I\:u$  is generated in a single degree for all monomials $u$ and has linear quotients with respect to the reverse lexicographic order of the generators.
\item[(d)] $I\:u$ has a linear resolution for all monomials $u$.
\item[(e)] $I\:u$ is   generated in a single degree for all monomials $u$.
\end{enumerate}
\end{Theorem}

\begin{proof}
(a)\implies (b): It is enough to show that for variable $x_i$, $I\:x_i$ is polymatroidal. Let $I=\sum_{j=0}^dI_jx_i^j$,
where for all $u\in G(I_j)$, $x_i\nmid u$. Then $I\:x_i=I_0+\sum_{j=1}^dI_jx_i^{j-1}$. Set $J=\sum_{j=1}^dI_jx_i^{j-1}$. Then
$I=I_0+x_iJ$, and $I\:x_i=I_0+J$. If $J=0$, then $I\:x_i=I_0=I$, and there is nothing to prove.

Now let $J\neq 0$.
We want to show that $I_0\subseteq J$. Let  $u$ be monomials with  $u\in I_0$. Since $J\neq0$ there exists a monomial $v\in I$ such that  $v\in x_iJ$.  Since $I$ is polymatroidal it satisfies the symmetric exchange property, see \cite[Theorem 12.4.1]{HH}.  Therefore, since  $x_i$ does not divide $u$ but does divide $v$, it follows that  there exists a variable $x_t$ with $t\neq i$  such that
$ux_i/x_t\in I$. Hence $ux_i/x_t\in x_iJ$, so $u/x_t\in J$. This  implies that $u\in J$. Thus we conclude that $I\:x_i=J$.

Let $u,v\in G(J)$ . So $x_iu,x_iv\in x_iJ\subseteq I$. If $\deg_{x_i}(u)=\deg_{x_i}(v)$,  since $I$ is polymatroidal,  it follows that $x_iu,x_iv$ satisfies  exchange property. Hence exchange property is satisfied for $u$ and $v$.

Let $\deg_{x_i}(u)>\deg_{x_i}(v)$, so $x_i|u$. Now for variable $x_l$ with $\deg_{x_l}(u)>\deg_{x_l}(v)$, we want to show that there exists variable $x_j$ such that $\deg_{x_j}(v)>\deg_{x_j}(u)$
and $(u/x_l)x_j\in G(J)$. Since $\deg_{x_l}(x_iu)>\deg_{x_l}(x_iv)$ and $I$ is polymatroidal, it follows that there exists variable $x_j$ such that $\deg_{x_j}(x_iv)>\deg_{x_j}(x_iu)$
and $(x_iu/x_l)x_j\in G(I)$. Since $x_i|u$, we have that $(x_iu/x_l)x_j\in I_tx_i^t$  for $t\geq 1$. Hence $(u/x_l)x_j\in I_tx_i^{t-1}\subseteq J$. Also we have that $\deg_{x_j}(v)>\deg_{x_j}(u)$.

(b)\implies (c): Any polymatroidal ideal is generated in a single degree   and has linear quotients with respect to the reverse lexicographic order of the generators, as shown in \cite[Lemma 1.3]{HT}. Therefore (b) implies (c) trivially.

(c)\implies (d) follows from the general fact that ideals generated in a single with  linear quotients have a linear resolution (see \cite[Lemma 4.1]{CH}), and (d)\implies (e) is trivial.

(e)\implies (a): Let $v,w\in G(I)$ with $\deg_{x_i}(v)>\deg_{x_i}(w)$. We want to show that there exists variable $x_j$
such that $\deg_{x_j}(w)>\deg_{x_j}(v)$ and $(v/x_i)x_j\in G(I)$. By assumption $I\:\frac{v}{x_i}$ is generated in a single degree. Hence, since $x_i\in G(I\:v/x_i)$ it follows that
$I\:v/x_i$ is generated in degree $1$. Hence, since $w/\gcd(w,v/x_i)\in  I\:v/x_i$, there exists $z\in G(I)$ such that $x_j=z/\gcd(z,v/x_i)$ for some $j$ and such that $x_j$ divides $w/\gcd(w,v/x_i)$. Then $\deg_{x_j}(w)>\deg_{x_j}(v/x_i)$. So since $\deg_{x_i}(v)>\deg_{x_i}(w)$ it follows  that $x_j\neq x_i$. Hence $\deg_{x_j}(w)>\deg_{x_j}(v/x_i)=\deg_{x_j}(v)$. Our  assumption (for $u=1$) implies that  $I$ is generated in a single degree.  Hence $\deg(z)=\deg(v)$. On the other hand,  it follows from   $x_j=z/\gcd(z,v/x_i)$ that $\deg_{x_l}(z) \leq \deg_{x_l}(v/x_i)=\deg_{x_l}(v)$ for all $l\neq i,j$ and $\deg_{x_j}(z)=\deg_{x_j}(v/x_i)+1=\deg_{x_j}(v)+1$ and also $\deg_{x_i}(z) \leq \deg_{x_i}(v/x_i)=\deg_{x_i}(v)-1$.  Therefore, $z=(v/x_i)x_j$.
\end{proof}

 We denote the set of monomial prime ideals of $S=K[x_1,\ldots,x_n]$ by $\P(S)$. Let $P\in \P(S)$ be a monomial prime ideal. Then $P=P_C$ for some subset $C\subset [n]$, where $P_C=(\{x_i\:\; i\not\in C\})$ and  $IS_P=JS_P$ where $J$ is the monomial ideal  obtained from $I$ by the substitution $x_i\mapsto 1$ for all $i\in C$. We call $J$ the monomial localization of $I$ with respect to $P$ and denote it by $I(P)$.

 For example, if $I=(x_1x_2x_3, x_2x_3x_4, x_3x_5x_6)\subset K[x_1,\ldots,x_6]$ and $C=\{4\}$, then $I(P_C)=(x_2x_3, x_3x_5x_6)$.

 Let $C\subset [n]$ and set $x_C=\prod_{i\in C}x_i$. Then $I(P_C)=I\: x_C^\infty =I\: x_C^k$ for $k$ large enough. In particular, if $I$ is a squarefree monomial ideal we have  that $I(P_C)=I\: x_C$. Therefore we obtain

 \begin{Corollary}
 \label{newyear}
 Let $I$ be a squarefree monomial ideal. The following conditions are equivalent:
 \begin{enumerate}
 \item[(a)] The ideal $I$ is a matroidal.
 \item[(b)] For all $P\in \P(S)$ the ideal  $I(P)$ is matroidal.
 \item[(c)] For all $P\in \P(S)$ the ideal  $I(P)$  is generated in a single degree and has linear quotients with respect to the reverse lexicographic order of the generators.
 \item[(d)] For all $P\in \P(S)$ the ideal $I(P)$ has  a linear resolution.
 \item[(e)] For all $P\in \P(S)$ the ideal $I(P)$  is   generated in a single degree.
 \end{enumerate}
 \end{Corollary}

 \begin{Corollary}
 \label{easyproof}
 Let $I$ be a squarefree monomial ideal. The following conditions are equivalent:
 \begin{enumerate}
 \item[(a)] The ideal $I$ is a matroidal.
 \item[(b)] For all $P\in \P(S)$ and all integers $k>0$ the ideal  $I^k(P)$ has a linear resolution.
 \item[(c)] For all $P\in \P(S)$ there exists an integer $k>0$ such that the ideal  $I^k(P)$ has a linear resolution.
 \item[(d)] For all $P\in \P(S)$ there exists an integer $k>0$ such that the ideal  $I^k(P)$ is  generated in a single degree.
 \item[(e)] For all $P\in \P(S)$ and all integers $k>0$ the ideal $I^k(P)$  is generated in a single degree.
 \end{enumerate}
 \end{Corollary}

 \begin{proof}
 (a)\implies (b): Since $I$ is a matroidal,  $I^k$ is polymatroidal for all $k$ (see \cite[Theorem 5.3]{CH}). Hence by \cite[Corollary 3.2]{HRV}, $I^k(P)$ is polymatroidal for all $P\in \P(S)$. So   $I^k(P)$ has a linear resolution for all $P\in \P(S)$ and all $k$.

 The implications (b)\implies (c) \implies (d), and (b)\implies (e)\implies (d) are trivial.

(d)\implies (a): By Corollary \ref{newyear} it is enough to show that $I(P)$ is generated in a single degree for all $P$.  By assumption we know that  $ (I(P))^k$ (which is equal to $I^k(P)$) is generated in a single degree. Thus, since $I(P)$ is a squarefree, the desired conclusion follows once we have shown  that  if $J$ is squarefree monomial ideal and $J^k$ is generated in a single degree, then $J$ is generated in a single degree as well. Let $s$ be  the  smallest degree of a generator of $J$  and assume  that there exists $v\in G(J)$ with $\deg(v)=t$, $t>s$.  Then our assumption implies that  $J^k$ is generated in degree $sk$. Since $v^k\in J^k$ and $\deg(v^k)=tk>sk$,  there exist $u_1,\ldots,u_k\in G(J)$ such that $\prod_{i=1}^ku_i$ divides $v^k$ and $\deg(u_i)=s$ for each $i=1,\ldots,k$. Then  $u_1$ divides $v^k$, so since $u_1$ and $v$ are squarefree monomials, it follows that $u_1$ divides $v$, a contradiction.
 \end{proof}

\section{Monomial localizations of polymatroidal ideals}
\label{polymatroidalsection}
One would expect that Corollary~\ref{newyear} remains true if we replace  in its  statements  ``matroidal" by ``polymatroidal". This is the case for the equivalence of (a) and (b). However the following example shows that (a) is not equivalent to (e) if we replace ``matroidal" by ``polymatroidal" in statement (a).

Indeed, let $I=(x_1^2, x_1x_2, x_3^2,x_2x_3)$. Then $I$ is not polymatroidal, but all monomial localizations are generated in a single degree. On the other hand, the ideal $I$ in this example does not have a linear resolution. So one may expect that polymatroidal ideals can be characterized by the properties (c) and  (d) of Corollary~\ref{newyear}.

In the following  special cases we can prove this.

\begin{Proposition}
\label{somayehbelievesit}
Let $I\subset K[x_1,\ldots,x_n]$ be a monomial ideal generated in degree $2$. Then the following conditions are equivalent:
\begin{enumerate}
\item[(a)] The ideal $I$ is a polymatroidal.
 \item[(b)] For all $P\in \P(S)$ the ideal  $I(P)$ is polymatroidal.
 \item[(c)] For all $P\in \P(S)$ the ideal  $I(P)$  is generated in a single degree and has linear quotients with respect to the reverse lexicographic order of the generators.
 \item[(d)] For all $P\in \P(S)$ the ideal $I(P)$ has  a linear resolution.
\item[(e)] After relabeling of the variables there exist integers $0\leq k\leq m\leq n$ such that
\[
I=((x_1,\ldots,x_k)(x_1,\ldots,x_m), J),
\]
where $J$ is a  squarefree monomial ideal in the variables $x_{k+1},\ldots,x_m$ satisfying the following property:
\[
\text{$(*)\hspace{0.5cm} $If  $x_ix_j\in J$ and  $k+1\leq l\leq m$ with $l\neq i,j$, then $x_ix_l\in J$ or $x_jx_l\in J$.}
\]
\end{enumerate}
\end{Proposition}

\begin{proof}
The implication (a)\implies (b)  is known (\cite[Corollary 3.2]{HRV}) and the implications (b)\implies (c) \implies (d) are known.

(d)\implies (e): After a relabeling of the variables we may assume the $x_i^2\in I$ if and only if $i\in[k]$. Suppose that $k\geq 2$ and let $1\leq i, j\leq k$ and $i\neq j$. Since $I$ is generated in degree 2 and has a linear resolution it  is known by \cite[Theorem 3.2]{HHZ} that $I$ has linear quotients with respect to a suitable order of the generators. We may assume that $x_i^2$ comes before $x_j^2$ in this order. Hence, since   $(x_i^2)\:x_j^2= (x_i^2)$,  there exists a monomial $u\in G(I)$ coming before $x_j^2$ such $(u)\:x_j^2=(x_i)$. It follows that $u=x_ix_j$. This shows that $(x_1,\ldots,x_k)^2\subset I$.

Let $\MI$ be the subset of elements $j\in [n]$ with the property that $j>k$ and $x_j|u$ for some $u\in G(I)$. After a relabeling of the variables $x_{k+1},\ldots,x_n$ we may assume that $\MI=\{k+1,\ldots,m\}$. Let $u=x_ix_j$ with $j\in \MI$ and $i\in[k]$. Then $x_i\in I(P_{\{j\}})$, and since $I(P_{\{j\}})$ has a linear resolution, all generators of  $I(P_{\{j\}})$ are of degree $1$. In particular, for any $t\in [k]$ we must have that $x_t \in G(I(P_{\{j\}}))$. This implies that $x_tx_j\in G(I)$. Thus we have shown that $(x_1,\ldots,x_k)(x_1,\ldots,x_m)\subset I$.

Let $J$ be the ideal generated by all $u\in G(I)$ which do not belong to the ideal  $(x_1,\ldots,x_k)(x_1,\ldots,x_m)$. Then  $J$ is a squarefree monomial ideal in the variables $x_{k+1},\ldots,x_m$. Let $x_ix_j\in J$ and $l$ an integer with  $k+1\leq l\leq m$ and $l\neq i,j$.

If $k=0$, then $x_lx_h\in J$ for some $h$ and  $J$ is matroidal by Corollary~\ref{newyear}.  Comparing $x_lx_h$ with $x_ix_j$ we see $x_ix_l\in J$ or $x_jx_l\in J$.

If $k>0$, then $x_1x_l\in I$. Therefore $x_1\in  I(P_{\{l\}})$, and hence  $I(P_{\{l\}})$ is generated in degree $1$, since it has a linear resolution. This implies that $x_ix_l\in J$ or $x_jx_l\in J$.

(e)\implies (a):  Let $u,v\in G(I)$. We have  to show that this pair satisfies the polymatroidal  exchange property. Since $(x_1,\ldots,x_k)(x_1,\ldots,x_m)$ is polymatroidal and $J$ is matroidal because of $(*)$, we may assume that $u\in  (x_1,\ldots,x_k)(x_1,\ldots,x_m)$ and $v\in J$.

Let $u=x_tx_l$ and $v=x_ix_j$, then the exchange property is satisfied because  $x_sx_i\in G(I)$ or $x_sx_j\in G(I)$ for all $s\neq i,j$, due to
$(*)$.
\end{proof}

For the proof of the next result we recall the following well-known fact.

\begin{Lemma}
\label{wellknown}
Let $J\subset S$ be a graded ideal with linear resolution and such that $\length (S/J)<\infty$. Then $J=(x_1,\ldots,x_n)^k$ for some $k$.
\end{Lemma}

\begin{proof}
Since $\length (S/J)<\infty$ it follows that  $\reg(S/J)=\max\{j\:\; (S/J)_j\neq 0\}$, see \cite[Lemma 1.1 ]{CH}. We may assume that  $J$ has a $k$-linear resolution. Therefore,   $\reg(S/J)=k-1$, and hence   $(S/J)_j=0$ for $j\geq k$. It follows that $J=(x_1,\ldots,x_n)^k$.
\end{proof}

\begin{Definition}
 {\em Given  positive integers $d, a_1,\ldots, a_n$. We let $I_{(d;a_1,\ldots,a_n)}\subset  S=K[x_1,\ldots,x_n]$ be the monomial ideal generated by the monomials $u\in S$ of degree $d$ satisfying  $\deg_{x_i}(u)\leq a_i$ for  all $i=1,\ldots,n$. Monomial ideals of this type are called ideals of {\em Veronese type}.}
\end{Definition}
Obviously, monomial ideals of Veronese type are polymatroidal.

\begin{Proposition}
\label{purepowers}
Let $I\subset K[x_1,\ldots,x_n]$ be a monomial ideal generated in degree $d$ and suppose that $I$ contains at least $n-1$ pure powers of the variables, say $x_1^d,\ldots,x_{n-1}^d$. Then the following conditions are equivalent:
\begin{enumerate}
\item[(a)] The ideal $I$ is a polymatroidal.
 \item[(b)] For all $P\in \P(S)$ the ideal $I(P)$ has  a linear resolution.
 \item[(c)] The ideals $I$ and $I(P_{\{n\}})$ have a linear resolution.
 \item[(d)] $I=I_{(d;d,\ldots,d,k)}$ for some $k$.
 \end{enumerate}

\end{Proposition}
\begin{proof}
The implication (a)\implies (b) is known and the implications (b)\implies (c) and (d)\implies (a) are trivial. Thus it remains to show that (c) implies (d).

To this end   we write
\[
I=I_0+I_1x_n+\cdots +I_kx_n^k,
\]
where $I_j$ is a monomial ideal in $S'=K[x_1,\ldots,x_{n-1}]$ for all $j$.

Several times in our proof we will apply the following fact, which is an immediate consequence of \cite[Theorem 2.1]{GHP}: let $J\subset S$ be a monomial ideal with linear resolution, and let $a_1, \ldots, a_n$ be positive integers. Then the monomial  ideal $J'$ generated by the monomials $u\in G(J)$ with $\deg_{x_i}u\leq a_i$ for $i=1,\ldots, n$ has linear resolution as well.  We refer to this result as to the `restriction lemma'.

Applying the restriction lemma to $I$ it follows that $I_0$ has a $d$-linear resolution. Our assumption implies that $x_1^d,\ldots x_{n-1}^d\in I_0$. In particular, it follows that $\length(S'/I_0)<\infty$. Thus Lemma~\ref{wellknown} implies that $I_0=\nn^d$ where $\nn=(x_1,\ldots,x_{n-1})$.

Next we show by induction on $j$ that $I_{k-j}=\nn^{d-k+j}$. For $j=0$, we have to show that $I_k=\nn^{d-k}$. Indeed, by assumption the ideal $I(P_{\{n\}})=I_0+I_1+\cdots +I_k$ has a linear resolution. Since $I_j$ is generated in degree $d-j$, it follows that $I(P_{\{n\}})=I_k$ and moreover, that $\nn^d=I_0\subset I_k$. Hence $I_k$ has a $(d-k)$-linear resolution and contains  $x_i^{d-k}$ for $i=1,\ldots,n-1$. Again applying Lemma~\ref{wellknown}, it follows that $I_k=\nn^{d-k}$. This completes the proof of the induction begin.

Now assume that $j>0$ (and $\leq k-1$), and assume that $I_{k-l}=\nn^{d-k+l}$ for $l=0,\ldots,j-1$. We set
\[
J=I_0+I_1x_n+\cdots +I_{k-j}x_n^{k-j} \quad \text{and}\quad L=\nn^{d-k+j-1}x_n^{k-j+1}+\cdots +\nn^{d-k}x_n^k.
\]
The ideal $L$ is polymatroidal, and hence has a $d$-linear resolution. Applying the restriction lemma to $I$ we see that $J$ has a $d$-linear resolution.
We have $$J\cap L=(I_0\cap L)+(I_1x_n\cap L)+\cdots +(I_{k-j}x_n^{k-j}\cap L)=$$
$$I_0x_n^{k-j+1}+I_1x_n^{k-j+1}+\cdots+I_{k-j}x_n^{k-j+1}=
(I_0+I_1+\cdots+I_{k-j})x_n^{k-j+1}.$$
So $\reg(J\cap L)\geq d+1$. On the other hand by  the exact sequence
\[
0\rightarrow J\cap L\rightarrow J\oplus L\rightarrow I\rightarrow 0
\]
we have that $\reg(J\cap L)\leq \max\{\reg(J\oplus L),\reg(I)+1\}=d+1$. Then $\reg(J\cap L)=d+1$. Hence $J\cap L=(I_0+I_1+\cdots+I_{k-j})x_n^{k-j+1}=I_{k-j}x_n^{k-j+1}$. So $I_{k-j}$ has a $(d-k+j)$-linear resolution and contains $x_i^{d-k+j}$ for $i=1,\ldots,n-1$, because $\nn^d=I_0\subset I_{k-j}$. By Lemma~\ref{wellknown},  $I_{k-j}=\nn^{d-k+j}$.
Altogether we have shown that $I=\nn^d+\nn^{d-1}x_n+\cdots +\nn^{d-k}x_n^k=I_{(d;d,\ldots,d,k)}$, as desired.
\end{proof}

\begin{Corollary}
\label{dimtwo}
Let $I\subset K[x_1,x_2]$ be a monomial ideal. The following conditions are equivalent:
\begin{enumerate}
\item[(a)] $I$ is polymatroidal.
\item[(b)]  For all $P\in \P(S)$ the ideal $I(P)$ has  a linear resolution.
\item[(c)] $I$ has a linear resolution.
\end{enumerate}
\end{Corollary}

\begin{proof}
The conditions (b) and (c) are equivalent, because $I(P)$ is a principal ideal for $P\neq (x_1,x_2)$, and the implication (a) \implies (b) is known. For the proof of the implication (b) \implies (a) we write $I=uJ$, where $u$ is the greatest common divisor of the generators of $I$. $I$ is polymatroidal if and only if $J$ is polymatroidal, and $I$ satisfies (b) if and only if $J$ does. So we assume from the very beginning that greatest common divisor of  the generators $I$ is $1$. This implies that $I$ contains a pure power of $x_1$ or a pure power of  $x_2$. Thus the desired conclusion follows from Proposition~\ref{purepowers}.
\end{proof}

\begin{Definition}{\em
Let $I$ be a monomial ideal. We say that $I$ satisfies the {\em strong exchange property} if $I$ is generated in a single degree and
for all $u, v\in G(I)$ and for all $i,j$ with $\deg_{x_i}(u)> \deg_{x_i}(v)$ and $\deg_{x_j}(u)< \deg_{x_j}(v)$, one has
$x_j(u/x_i)\in I$.}
\end{Definition}

\begin{Proposition}
\label{dimthree}
Let $I\subset S=K[x_1,x_2,x_3]$ be a monomial ideal. The following conditions are equivalent:
\begin{enumerate}
\item[(a)] $I$ is polymatroidal.
\item[(b)] $I$ is polymatroidal satisfying the strong exchange property.
\item[(c)]  For all $P\in \P(S)$ the ideal $I(P)$ has  a linear resolution.
\end{enumerate}
\end{Proposition}

\begin{proof}
The implications (b)\implies (a) and (a) \implies (c) are known. This it remains to be shown that (c) implies (b). Let $I=uJ$ where $u$ is the greatest common divisor of the generators of $I$. It is known \cite[Theorem 1.1]{HHV} that $I$ is polymatroidal satisfying the strong exchange property, if and only if $J$ is of Veronese type. Since $I(P)$  has  a linear resolution for all $P\in P(S)$ if and only if the same holds true for all $J(P)$, we may assume from the very beginning that $u=1$, and then have to show that $I$ is of Veronese type. Let $a_i=\max\{\deg_{x_i}(u)\:\; u\in G(I)\}$ for $i=1,\ldots,3$. We claim that $I=I_{(d;a_1,a_2,a_3)}$ where $d$ is the common degree of the generators of $I$. We first show that for each $i$  the set of monomials
\[
\mathcal{A} =\{u\in K[x_1,x_2,x_3]\:\; \deg(u)=d, \deg_{x_i}(u)=a_i \text{ and } \deg_{x_j}(u)\leq a_j \text{ for } j\neq i \}
\]
belongs to $I$.

Indeed, (c) implies that $I(P_{\{i\}})$ is generated by the monomials $v\in K[x_j,x_k]$ such that $vx_i^{a_i}\in I$ and has a linear resolution. Therefore, by Corollary~\ref{dimtwo}, $I(P_{\{i\}})$ is polymatroidal. Hence there exist numbers $0\leq e\leq f\leq d-a_i$ such that
\[
I(P_{\{i\}})=(x_j^rx_k^s\:\; r+s=d-a_i, r\leq a_j, s\leq a_k \text{ and }e\leq r\leq f).
\]
Assume now that $\mathcal{A}\not\subset I$. Then it follows $e>0$ or $f<d-a_i$. We may assume that $e>0$. Therefore, $x_k^{d-a_i}x_i^{a_i}\not \in I$. On the other hand, since the greatest common divisor of the elements of $G(I)$ is equal to $1$, it follows that there exists monomial $x_k^{d-b}x_i^{b}\in I$  with $b<a_i$. Hence $x_k^{d-b}\in I(P_{\{i\}})$, a contradiction because $I(P_{\{i\}})$ does not contain a pure power of $x_k$.

In order to complete the proof of the claim, we introduce the following ideals $J_{b_1,b_2,b_3}$ with $a_i\leq b_i\leq d$ for $i=1,2,3$. The ideal  $J_{b_1,b_2,b_3}$ is generated by all generators of $I$ and all monomials $x_1^{r_1}x_2^{r_2}x_3^{r_3}$ of degree $d$ such that $r_j\leq b_j$ for all $j$ and there exists $i\in[3]$ with $a_i\leq r_i\leq b_i$. We will show by induction on $b_1+b_2+b_3$  that  $J_{b_1,b_2,b_3}$  has a linear resolution for all $b_i$. In particular,  $J_{d,d,d}$ has a linear resolution.  Hence by  Lemma~\ref{wellknown},  $J_{d,d,d}=(x_1,x_2,x_3)^d$  since  $J_{d,d,d}$  contains the pure powers $x_i^d$. This then implies that $I=I_{(d;a_1,a_2,a_3)}$.

The induction begin with $b_1+b_2+b_3=a_1+a_2+a_3$ is trivial because in that case  $a_i=b_i$ and $J_{a_1,a_2,a_3}=I$, which by assumption has a linear resolution. Now assume that $b_1+b_2+b_3>a_1+a_2+a_3$. Then $b_i>a_i$ for some $i$, say for $i=1$. By induction hypothesis the ideal $H=J_{b_1-1,b_2,b_3}$ has a $d$-linear resolution. Let $J=J_{b_1,b_2,b_3}$, and consider the exact sequence
\[
0\To H \To J\To J/H\To 0.
\]
The module $J/H$ is annihilated by $x_2$ and $x_3$. Therefore, $J/H$ is an $S/(x_2,x_3)$-module generated by the residue classes of the elements $vx_1^{b_1}$ with $v\in K[x_2,x_3]$ of degree $d-b_1$. Since no power of $x_1$ annihilates the generators of $J/H$ it follows that $J/H$ is a free $S/(x_2,x_3)$. It follows that $J/H$ has a $d$-linear resolution. Therefore we conclude from the above exact sequence that $J$ has a $d$-linear resolution.
\end{proof}

\begin{Proposition}
\label{anothercase}
Let $I\subset S=K[x_1,\ldots,x_n]$ be a monomial ideal with no embedded prime ideals such that $I(P)$ has a linear resolution for all $P\in V^*(I)$, and let $\Ass(S/I)=\{P_1,\ldots,P_r\}$. Let $\mm=(x_1,\ldots,x_n)$ be the graded maximal ideal of $S$.  Then the following holds:
\begin{enumerate}
\item[(a)] If $P_i+P_j=\mm$ for all $i\neq j$, then $I$ is polymatroidal.
\item[(b)] If  $r\leq 2$, then  $I$ is a transversal polymatroidal ideal.  If $r=3$, then either $I$ is again  a transversal polymatroidal ideal or $I$ is a matroidal ideal generated in degree $2$ of the form $I=P_1\sect P_2\sect P_3$ such that  $\Sect_{i=1}^3G(P_i)=\emptyset$ and  $G(P_i)\union G(P_j)=\{x_1,\ldots,x_n\}$ for all $i\neq j$. 
\item[(c)] If $\height(I)=n-1$, then $I$ is polymatroidal.
\end{enumerate}
\end{Proposition}

\begin{proof}
Let $P\in\Ass(S/I)$. Since $I$ is a monomial ideal with no embedded prime ideals, it follows that $P$ is a minimal prime ideal of $I$. Therefore,  $\length(S(P)/I(P))<\infty$. Since $I(P)$ has a linear resolution, it follows from Lemma \ref{wellknown}, that $I(P)=P^k$ for some $k$.
Therefore $I=P_1^{a_1}\cap\cdots\cap P_r^{a_r}$.

(a) Since $I$ is generated in a single degree and $P_i+P_j=\mm$ for all $i\neq j$, it follows from  a result of Francisco and Van Tuyl \cite[Theorem 3.1]{FV} that $I$ is polymatroidal.

(b) If $r=1$, then $I=P_1^{a_1}$ is a transversal  polymatroidal.

If $r=2$, then $I=P_1^{a_1}\sect P_2^{a_2}$. Since $I$ is generated in a single degree we conclude that $G(P_1)\sect G(P_2)=\emptyset$. Therefore, $I=P_1^{a_1}P_2^{a_2}$, and the assertion follows.

Now let $r=3$, then $I=P_1^{a_1}\cap P_2^{a_2}\cap P_3^{a_3}$. We may assume that $I$ is full supported, i.e., $\{x_1,\ldots,x_n\}= \Union_{u\in G(I)}\supp(u)$.

First  assume that $P_i\nsubseteq P_j+P_k$ for all $i,j,k$.  Then,  since $I(P_j+P_k)=P_j^{a_j}\cap P_k^{a_k}$ is generated in a single degree, it follows that $G(P_j)\cap G(P_k)=\emptyset$ for $j\neq k$. Hence $I=P_1^{a_1} P_2^{a_2} P_3^{a_3}$
is a transversal polymatroidal ideal.

Next we may assume that $P_1\subseteq P_2+P_3$. In particular,  $P_2+P_3=\mm$, since $I$ is full supported. We claim that $P_i+P_j=\mm$ for all $i\neq j$ and hence by part (a), $I$ is polymatroidal. It remains  to be shown that $P_1+P_2=\mm$ and $P_1+P_3=\mm$.
Assume that $P_1+P_2\neq\mm$ and set $P=P_1+P_2$. Then $I(P)=P_1^{a_1}\cap P_2^{a_2}$. Since $I(P)$ is generated in a single degree, we have that $G(P_1)\cap G(P_2)=\emptyset$. So since
$P_1\subseteq P_2+P_3$, it follows that $P_1\subseteq P_3$, a contradiction.
Therefore $P_1+P_2=\mm$. Similarly we can see that $P_1+P_3=\mm$.

Now we want to show that  $G(P_i)\cap G(P_j)\nsubseteq G(P_k)$ for distinct $i$, $j$ and $k$. Assume  $G(P_i)\cap G(P_j)\subseteq G(P_k)$ for some $i,j$ and $k$. Let $x_\ell$ be a variable.  If $x_\ell\in G(P_i)\cap G(P_j)$, then $x_\ell\in G(P_k)$, and if $x_\ell\not\in G(P_i)\cap G(P_j)$, then we may assume that  $x_\ell\not\in G(P_i)$. In that case it follows that  $x_\ell\in G(P_k)$,  since $P_i+P_k=\mm$.  Therefore $P_k=\mm$, a contradiction.

Now we claim that $a_1=a_2=a_3$.
We may assume that $a_1\geq a_2\geq a_3$ and  that $I$ is generated in degree $d$. Let $x_i\in G(P_1)\cap G(P_2)\setminus G(P_3)$ and $x_j\in G(P_3)\setminus G(P_1)$. Then since $a_1\geq a_2$, it follows that $x_i^{a_1}x_j^{a_3}\in I$. So there exist integers $s\leq a_1$ and $t\leq a_3$ such that $x_i^sx_j^t\in G(I)$. Since
$x_i^sx_j^t\in P_3^{a_3}$ and $x_i\not\in P_3$, we have  $x_j^t\in P_3^{a_3}$, and so $t=a_3$. On the other hand, since $x_i^sx_j^t\in P_1^{a_1}$ and $x_j\not\in P_1$,
 it follows that $x_i^s\in P_1^{a_1}$, and so $s=a_1$. Hence $x_i^{a_1}x_j^{a_3}\in G(I)$. Therefore $d=a_1+a_3$. Now let $x_i\in G(P_1)\cap G(P_3)\setminus G(P_2)$ and $x_j\in G(P_2)\setminus G(P_1)$.
Then similarly $x_i^{a_1}x_j^{a_2}\in G(I)$, so $d=a_1+a_2$. Therefore  $a_2=a_3$. Set $a=a_2=a_3$. Next we show that $a_1<2a$. Assume $a_1\geq2a$. Let $x_i\in G(P_1)\cap G(P_2)$
and $x_j\in G(P_1)\cap G(P_3)$, then $x_i^{a_1-a}x_j^a\in I$. Hence $a_1=\deg(x_i^{a_1-a}x_j^a)\geq d=a_1+a$, so $a\leq 0$, a contradiction.
Now let $x_i\in G(P_1)\cap G(P_2)$, $x_j\in G(P_1)\cap G(P_3)$ and $x_k\in  G(P_3)$. Then  $x_i^ax_j^{a_1-a}x_k^{2a-a_1}\in I$. Therefore,
$2a=\deg(x_i^ax_j^{a_1-a}x_k^{2a-a_1})\geq d=a_1+a$, hence $a\geq a_1$, and so $a_1=a$.
 
Now we have  $I=P_1^a\cap P_2^a\cap P_3^a$. We claim  that $a=1$. The claim  implies that $I=P_1\cap P_2\cap P_3$. Hence, since $I$ is generated in a single degree we conclude that $G(P_1)\cap G(P_2)\cap G(P_3)=\emptyset$.
 
In order to prove the claim, assume to contrary that  $a>1$. Let $x_i\in G(P_1)\cap G(P_2)$, $x_j\in G(P_1)\cap G(P_3)$ and $x_k\in G(P_2)\cap G(P_3)$. Then $x_i^{a-1}x_j^{a-1}x_k\in I$, because $x_i^{a-1}x_j^{a-1}\in P_1^a$, $x_i^{a-1}x_k\in P_2^a$ and $x_j^{a-1}x_k\in P_3^a$.
So $2a-1=\deg(x_i^{a-1}x_j^{a-1}x_k)\geq d=2a$, a contradiction.

(c) If $r=1$, then  $I=P_1^{a_1}$ is  polymatroidal, and if $r>1$, the assertion follows from (a).
\end{proof}

Based on Proposition~\ref{somayehbelievesit},  Proposition~\ref{purepowers}, Corollary~\ref{dimtwo}, Proposition~\ref{dimthree} and Proposition~\ref{anothercase} and based on experimental evidence we are inclined   to make the following

\begin{Conjecture}
\label{mustbetrue}{\em
A monomial ideal $I$ is polymatroidal if and only if $I(P)$ has a linear resolution for all monomial prime ideals $P$.}
\end{Conjecture}

\medskip
The following examples show  that the localization condition of Conjecture~\ref{mustbetrue}  can not be weakened.

\begin{Example}
\label{threexamples}
{\em (a) The ideal  $I=(x_1x_3^2,x_1^2x_3,x_1x_2x_3, x_2^2x_3)$  and all $I\: x_i$ have a linear resolution, but $I$ is not polymatroidal.

(b)  The ideal  $I=(x_1^3,x_1^2x_2, x_1^2x_3,x_2x_3x_4, x_1x_2x_3, x_1x_3x_4, x_1^2x_4)$ and all $I(P_{\{i\}})$   have a linear resolution, but $I$ is not polymatroidal.

(c) The ideal $I=(x_1^3,x_1^2x_2,x_1^2x_3, x_2^3,x_1x_2^2,x_2^2x_3, x_3^3,x_1x_3^2,x_2x_3^2)$ has  linear relations, and all $I(P_{\{i\}})$
are polymatroidal, but $I$ is not polymatroidal. }
\end{Example}

\medskip
For the proof of Proposition~\ref{anothercase}(c) one could skip the assumption that $I$ has no embedded components, if one could prove the following statement: $(*)$ Let $I\subset S$ be a monomial ideal with linear resolution and such that $I\mm$ is polymatroidal. Then $I$ is polymatroidal.

Indeed, assuming $(*)$ the following can be shown: Let $I=J\sect Q$ and assume that $I$ has  a linear resolution, $J$ is componentwise polymatroidal and $Q$ is $\mm$-primary, then $I$ is polymatroidal. To see this, observe that $I\mm^{j-d}= I_{\langle j\rangle}=J_{\langle j\rangle}$ for $j\gg 0$, where $d$ is the degree of the generators of $I$. Here, for any graded ideal $L$, we denote  by $L_{\langle j\rangle}$ the ideal generated by the $j$th graded component of $L$.  Since $J$ is componentwise polymatroidal it follows that $I\mm^{j-d}$ is polymatroidal. The assertion now follows by induction on $j-d$ and by using $(*)$.

\medskip
Observe  that $(*)$ holds if our Conjecture~\ref{mustbetrue} is satisfied, because $I(P)=(I\mm)(P)$ for all $P\neq \mm$.

\medskip
We believe that if $I$ is a polymatroidal ideal generated in degree $d$, then $(I:\mm)_{\langle d-1\rangle}$ is polymatroidal. It can be shown that this is the case at least when $I$ is a polymatroidal ideal satisfying the strong exchange property. Assuming this is true in general, the above condition $(*)$ follows, because $I=I\mm:\mm$, if $I$ has a linear resolution. Obviously we have $I\subseteq I\mm:\mm$. Assume the inclusion is strict. Then there exists a homogeneous element   $f\in  I\mm:\mm\setminus I$. Thus the residue class  of $f$ in $S/I$ is a non-zero socle element of $S/I$. Say,  $I$ has a $d$-linear resolution.  Then it follows that $\deg(f)=d-1$. On the other hand, $I\mm$ has $(d+1)$-linear resolution. Therefore $I\mm:\mm$ is generated in degree $\geq d$, a contradiction since $f\in I\mm:\mm$.

\medskip
Note that  our conjecture is equivalent to the following statement: let $I$ be monomial ideal with linear resolution. Then  $I$ is polymatroidal if and only if  $I(P_{\{i\}})$ is polymatroidal for all $i$. We prove this version of Conjecture~\ref{mustbetrue} under additional assumptions.

\begin{Proposition}
\label{biequalsai}
Let $I$ be a monomial ideal with $d$-linear resolution, and assume that $I(P_{\{i\}})=I_{(d-a_i; a_1,\ldots,a_{i-1},a_{i+1},\ldots, a_n)}$ for $i=1,\ldots,n$. Then $I=I_{(d; a_1,\ldots, a_n)}$.
\end{Proposition}

\begin{proof}
For $k=1,\ldots,n$, let $I_k=x_k^{a_k}I(P_{\{k\}})$ and set $J=\sum_{k=1}^nI_k$. Then $J\subseteq I$. We first show that $(J:\mm^\infty)_{\langle d\rangle}=I$. In fact, by the definition of $J$ it follows that $(I/J)_{x_k}=0$ for $k=1,\ldots,n$. Therefore,  $I/J$ is a module of finite length, and hence we get
\begin{eqnarray}
\label{stupidr}
I\subseteq (J:\mm^\infty)_{\geq d}\subseteq (I:\mm^\infty)_{\geq d}.
\end{eqnarray}
Here for any graded ideal $L$ we set $L_{\geq d}=\Dirsum_{i\geq d}L_i$.

Since  $I$  has $d$-linear resolution, it follows that  $(I:\mm^\infty)_{\geq d}=I$. Indeed, our assumption on $I$ implies that $I:\mm^\infty=I+H$ where $H$ is generated in degree $\leq d-1$. This follows from \cite[Corollary 20.19]{Eis}. Thus  $(I:\mm^\infty)_{\geq d}= I+\mm H_{\langle d-1\rangle}$. Since $\mm H_{\langle d-1\rangle}\subset I$, the desired conclusion follows. Thus in combination with (\ref{stupidr}) we see that $(J:\mm^\infty)_{\geq d}=I$. Since $I$ is generated in degree $d$, we even get $(J:\mm^\infty)_{\langle d\rangle}=I$.

Now we want to show that $(J:\mm^\infty)_{\langle d\rangle}=I_{(d; a_1,\ldots, a_n)}$. Let $u\in(J:\mm^\infty)_{\langle d\rangle}$ such that
$\deg(u)=d$ and $u\mm^r\subseteq J\subseteq I_{(d; a_1,\ldots, a_n)}$ for some integer $r\geq 0$. Then $ux_i^r\in I_{(d; a_1,\ldots, a_n)}$ for all $i\in[n]$. Hence
for all $i\in [n]$ there exists $v_i\in G(I_{(d; a_1,\ldots, a_n)})$ such that $v_i|ux_i^r$. Therefore, $\deg_{x_j}(v_i)\leq \deg_{x_j}(u)$ for all $j\neq i$. Since  $\deg(u)=\deg(v_i)=d$, it follows that $\deg_{x_i}(u)\leq \deg_{x_i}(v_i)\leq a_i$. This shows that $u\in I_{(d; a_1,\ldots, a_n)}$.  Hence we proved that $(J:\mm^\infty)_{\langle d\rangle}\subseteq I_{(d; a_1,\ldots, a_n)}$.

Now let $x_1^{b_1}\cdots x_n^{b_n}\in G(I_{(d; a_1,\ldots, a_n)})$, then $\sum_{i=1}^nb_i=d$ and $b_i\leq a_i$ for all $i\in[n]$. We claim that $x_1^{b_1}\cdots x_n^{b_n}\mm^s\subseteq J$ with $s=\sum_{i=1}^na_i-d$.
Let $x_1^{c_1}\cdots x_n^{c_n}\in G(\mm^s)$. Then $x_1^{b_1+c_1}\cdots x_n^{b_n+c_n}\in x_1^{b_1}\cdots x_n^{b_n}\mm^s$. If $b_i+c_i<a_i$ for all $i$, then $\sum_{i=1}^na_i=d+s=\sum_{i=1}^nb_i+\sum_{i=1}^nc_i<\sum_{i=1}^na_i$, a contradiction. Hence for convenience we may  assume that $b_1+c_1\geq a_1$, and show that $x_1^{b_1+c_1}\cdots x_n^{b_n+c_n}\in I_1=x_1^{a_1}I_{(d-a_1; a_2,\ldots, a_n)}$. Since  $b_1+c_1\geq a_1$, it is enough to show that $x_2^{b_2+c_2}\cdots x_n^{b_n+c_n}\in I_{(d-a_1; a_2,\ldots, a_n)}$.

We may assume that $b_i+c_i>a_i$ for $i=2,\ldots, t$  and $b_i+c_i\leq a_i$ for $i=t+1,\ldots, n$ with $1\leq t\leq n$.
Since $b_i\leq a_i$ for all $i$, it follows that
\[
\sum_{i=1}^ta_i+\sum_{i=t+1}^n(b_i+c_i)=\sum_{i=1}^ta_i+d-\sum_{i=1}^tb_i+\sum_{i=t+1}^nc_i=\sum_{i=1}^t(a_i-b_i)+\sum_{i=t+1}^nc_i+d\geq d.
\]
Hence $\sum_{i=2}^ta_i+\sum_{i=t+1}^n(b_i+c_i)\geq d-a_1$. This  implies that
\[
x_2^{a_2}\cdots x_t^{a_t}x_{t+1}^{b_{t+1}+c_{t+1}}\cdots x_n^{b_n+c_n}\in I_{(d-a_1; a_2,\ldots, a_n)}.
 \]
Therefore $x_2^{b_2+c_2}\cdots x_n^{b_n+c_n}=w(x_2^{a_2}\cdots x_t^{a_t}x_{t+1}^{b_{t+1}+c_{t+1}}\cdots x_n^{b_n+c_n})\in I_{(d-a_1; a_2,\ldots, a_n)}$, as desired.
\end{proof}

\section{Componentwise polymatroidal ideals}
\label{extension}

In this section we extend the notion of polymatroidal ideals to monomial ideals which are not necessarily generated  in a single degree.

Let $I$ be a monomial ideal. We denote by $I_{\langle j\rangle}$ the monomial ideal generated by all monomial of degree $j$ in $I$. The ideal $I$ is called {\em componentwise linear}, if  $I_{\langle j\rangle}$ has a linear resolution for all $j$. Basic properties about componentwise linear ideals can be found in \cite{HH}.

\begin{Definition}
{\em Let $I$ be a monomial ideal. We say that $I$ is {\em componentwise polymatroidal}, if $I_{\langle j\rangle}$ is polymatroidal for all $j$.}
\end{Definition}

Observe that if  $d$ is the highest degree of a generator of $I$, then $I$ is componentwise polymatroidal if and only if $I_{\langle j\rangle}$ is polymatroidal for all $j\leq d$. Indeed, $I_{\langle j\rangle}=I_{\langle d\rangle}\mm^{j-d}$ for $j\geq d$. Moreover, all powers of $\mm$ are polymatroidal and products of polymatroidal ideals are again polymatroidal, see \cite[Theorem 5.3]{CH}

\medskip
 It is easy to see that  $I$ is componentwise polymatroidal if and only if $I\: u$ is componentwise polymatroidal for all monomials $u$. However if we only assume that  $I\: u$  is componentwise linear for all monomials $u$, it does not necessarily  follow that $I$ is componentwise polymatroidal. Indeed, let  $I=(x_1x_2,x_1x_3^2,x_2x_3^2).$
Then $I\:u $ is componentwise linear for all monomials $u$, but $I$ is not componentwise polymatroidal.

\medskip

It is natural to ask whether powers of componentwise  polymatroidal ideals are again componentwise  polymatroidal. There is a positive answer to this question in the following case.

\begin{Proposition}
\label{veryspecial}
Let $I$ be a componentwise polymatroidal ideal generated in at most 2 degrees. Then $I^k$ is componentwise polymatroidal for all $k$.
\end{Proposition}

\begin{proof} The statement is trivial if $I$ is generated in a single degree. So now assume that $I$ is generated in 2 degrees, say, in degree $d$ and $d+t$ with $t>0$. Then $I=I_{\langle d\rangle} +I_{\langle d+t\rangle}$. Hence
\begin{eqnarray}
\label{idealsum}
I^k=  \sum_{j=0}^k(I_{\langle d\rangle})^{k-j}(I_{\langle d+t\rangle})^{j}.
\end{eqnarray}
Since $I^k$ is generated in degree $\geq dk$ it remains to be shown that $(I^k)_{\langle kd+r\rangle}$ is polymatroidal for all $r\geq 0$.  It follows from
(\ref{idealsum}) that
\[
(I^k)_{\langle kd+r\rangle}=  \sum_{j=0}^\ell (I_{\langle d\rangle})^{k-j}(I_{\langle d+t\rangle})^{j}\mm^{r-tj},
\]
where  $\ell =\min\{k,\lfloor r/t\rfloor\}$.

Observe that for $j< \ell$ we have
\begin{eqnarray*}
(I_{\langle d\rangle})^{k-j}(I_{\langle d+t\rangle})^{j}\mm^{r-tj}&=&(I_{\langle d\rangle})^{k-j-1}(I_{\langle d+t\rangle})^{j}I_{\langle d\rangle}\mm^t\mm^{r-t(j+1)}\\
&\subseteq &(I_{\langle d\rangle})^{k-(j+1)}(I_{\langle d+t\rangle})^{j+1}\mm^{r-t(j+1)}.
\end{eqnarray*}
It follows that $(I^k)_{\langle kd+r\rangle}=(I_{\langle d\rangle})^{k-\ell}(I_{\langle d+t\rangle})^{\ell}\mm^{r-t\ell}$. Since products of polymatroidal ideals are polymatroidal the desired conclusion  follows.
\end{proof}

In general powers of componentwise polymatroidal ideals are not componentwise polymatroidal.

\begin{Example}
\label{polymatroidal}
{\em Let $I=(x_1^2, x_2^2x_3,x_1x_2x_3, x_1x_2^2, x_1x_3^3, x_2x_3^3)$. By using Proposition~\ref{dimthree} it is easy to see that  $I$ is componentwise polymatroidal. However $(I^2)_{\langle 6\rangle}$ is not polymatroidal, because  $(I^2)_{\langle 6\rangle}(P_{\{3\}})=(x_1x_2^3,x_2^4,x_1^2x_2,x_1^3)$ is not generated in a single  degree.}
\end{Example}

One would expect that componentwise polymatroidal ideals can also be  characterized by an exchange property of its minimal set of monomial generators. Suppose for a monomial ideal $I$  we require that  for all monomials $u,v\in G(I)$ the following condition holds: $(*)$ if $\deg_{x_i}(u)>\deg_{x_i}(v)$ for some $i$, then there exists an integer $j$ such that $\deg_{x_j}(v)>\deg_{x_j}(u)$ and $x_j(u/x_i)\in I$. Then it is easily checked that $I$ is necessarily generated in a single degree and hence polymatroidal.

Therefore  we give the following

\begin{Definition}
{\em Let $I$ be a monomial ideal. We say that $I$ satisfies the {\em non-pure exchange property},  if for all $u,v\in G(I)$ with $\deg(u)\leq\deg(v)$ and for all $i$ such that $\deg_{x_i}(v)>\deg_{x_i}(u)$, there exists $j$ such that
 $\deg_{x_j}(v)<\deg_{x_j}(u)$ and $x_j(v/x_i)\in I$.}
\end{Definition}

\begin{Proposition}
\label{exchange}
If $I$ is componentwise polymatroidal, then $I$ has the non-pure exchange property.
\end{Proposition}

\begin{proof}
Let $u,v\in G(I)$ with $\deg(u)\leq \deg(v)=t$ and  $\deg_{x_i}(v)>\deg_{x_i}(u)$ for some $i$. We may assume that  $\deg(u)< \deg(v)$, since $I_{\langle t\rangle}$ is polymatroidal.  By using the fact that  $u$ does not divide $v$, it follows that
there exists $l\neq i$ such that
\begin{eqnarray}
\label{star}
\deg_{x_l}(v)<\deg_{x_l}(u).
\end{eqnarray}
Since  $\deg(u)<\deg(v)$, there exists integer $a$ such that
$\deg(ux_l^a)=\deg(v)$. Then there exists $j$ such that
\begin{eqnarray}
\label{starstar}
\deg_{x_j}(v)<\deg_{x_j}(ux_l^a),
\end{eqnarray}
since $I_{\langle t\rangle}$ is polymatroidal and since   $\deg_{x_i}(v)>\deg_{x_i}(u)=\deg_{x_i}(ux_l^a)$. Moreover,   $x_j(v/x_i)\in I$.
If $j=l$, then by (\ref{star}),  $\deg_{x_j}(v)<\deg_{x_j}(u)$ and $x_j(v/x_i)\in I$. If $j\neq l$,  then (\ref{starstar}) implies
that  $\deg_{x_j}(v)<\deg_{x_j}(ux_l^a)=\deg_{x_j}(u)$ and $x_j(v/x_i)\in I$.
\end{proof}

Unfortunately, the converse of Proposition \ref{exchange} is not true. Indeed, let $I=(x_1x_2,x_1x_3^2,x_2x_3^2).$
Then $I$ has the non-pure exchange property but $I_{\langle 3\rangle}$ is not polymatroidal. On the other hand,  $I$ has linear quotients. Thus the question arises whether any monomial ideal satisfying the  non-pure exchange property has linear quotients. In view of Proposition~\ref{exchange} a positive answer to this question would imply that any componentwise polymatroidal ideal has linear quotients. In the following we show that ideals which are componentwise of Veronese type have linear quotients.

The following concept is needed for the next results: Let $I\subset J$ be monomial ideals with $G(I)\subset G(J)$. We say that $I$ can be {\em extended by linear quotients} to $J$, if the set $G(J)\setminus G(I)$ can be ordered $v_1,\ldots,v_m$  such that $(G(I),v_1, \ldots,v_i)\: v_{i+1}$ is generated by variables for $i=1,\ldots,m-1$. In a particular a monomial ideal $L$ has linear quotients, $(0)$ can be extended to $L$ by linear quotients.

It is known (\cite[Corollary 2.8]{SZ}) that an ideal with linear quotients is componentwise linear. In particular, if $I$ has linear quotients and $I$ can be extended to $J$ by linear quotients, then $J$ has linear quotients and hence a linear resolution.

\begin{Theorem}
\label{extend}
Let $I$ be an ideal of Veronese type generated in degree $d$, and  $J$ an ideal of Veronese type generated in degree $d+1$ such that $I\mm\subseteq J$.  Then $I\mm$ can be extended by linear quotients to $J$.
\end{Theorem}

\begin{proof}
Let $I=I_{(d;a_1,\ldots,a_n)}$. In the first step of the proof we assume that $J=I_{(d+1;a_1+1,\ldots,a_n+1)}$.
Let $u=x_1^{h_1} \cdots x_n^{h_n} \in G(J)$.
We define the set $$S_u=\{i\in [n]\;|\;h_i=a_i+1\}$$
and  the monomial $\bar{u}=\prod_{i\in S_u}x_i^{h_i}$.

\medskip
Now we consider the following order for elements of $G(J)\setminus G(I\mm)$:  we say that $u>v$, if either $|S_u|<|S_v|$,  or  $|S_u|=|S_v|$ and $\bar{u}>_{\lex}\bar{v}$, or  $|S_u|=|S_v|$, $\bar{u}=\bar{v}$ and $u>_{\lex}v$. We also set $u>v$ for all $v\in G(J)\setminus G(I\mm)$
and all $u\in G(I\mm)$.

We claim that with this order,  $I\mm$ can be extended to $J$ by linear quotients.
We have  to show that for all $u=x_1^{h_1}\cdots x_n^{h_n}\in G(J)$ and all $v=x_1^{t_1}\cdots x_n^{t_n}\in G(J)\setminus G(I\mm)$  with $u>v$  there exists $w\in G(J)$ with $w>v$  such that $(w)\:v=(x_j)$ and $x_j$ divides $u/\gcd(u,v)$. We distinguish several cases.

Case (a):  $u\in G(I\mm)$ and $v\in G(J)\setminus G(I\mm)$.  Since $u\in G(I\mm)$,  there exists $r\in [n]$ such that  $h_j\leq a_j$  for $j\neq r$. On the other hand,  since $v\in G(J)\setminus G(I\mm)$,  there exists $l\in[n]$ such that $t_l=a_l+1$. If there exists $p\neq r$ such that $x_p$ divides $u/\gcd(u,v)$, then $t_p<h_p\leq a_p$. Let $w=(v/x_l)x_p$; then $(w)\:v=(x_p)$ and $w\in G(J)$ with $w>v$, because $|S_w|<|S_v|$.  Next we consider the case that $x_p$ does not divide $u/\gcd(u,v)$ for all  $p\neq r$.  Then $(u)\:v=(x_r^c)$ for some integer $c$. If the $c=1$, then there is nothing to prove.  Otherwise,    $t_r+1<h_r\leq a_r+1$. Let  $w=(v/x_l)x_r$; then  $(w)\:v=(x_r)$ and $w\in G(J)$ with $w>v$, because $|S_w|<|S_v|$.

Case (b):  $u,v\in G(J)\setminus G(I\mm)$ and $|S_u|<|S_v|$. Since  $v\in G(J)\setminus G(I\mm)$, it follows that there exists $l\in[n]$ such that $t_l=a_l+1$. If there exists $r\in S_u$ with
$t_r<a_r$, we set  $w=(v/x_l)x_r$. Then $(w)\: v=(x_r)$  and $w\in G(J)$ with $w>v$, because $|S_w|<|S_v|$.
Next we consider the case that $t_r\geq a_r$ for all $r\in S_u$. Since $\deg(u)=\deg(v)$ and $|S_u|<|S_v|$, it follows that there exists $s\in [n]\setminus S_u$ such that $h_s>t_s$. We set $w=(v/x_l)x_s$. Then again $(w)\: v=(x_s)$,  and $w\in G(J)$ with $w>v$, because $|S_w|<|S_v|$.

Case (c):  $u,v\in G(J)\setminus G(I\mm)$, $|S_u|=|S_v|$ and $\bar{u}>_{\lex}\bar{v}$. There exist $l,r$ such that $r<l$, $h_r=a_r+1>t_r$ and $h_l<t_l=a_l+1$.
Let $w=(v/x_l)x_r$; then $(w)\: v=(x_r)$ and $w\in G(J)$ with $w>v$. Indeed, if  $|S_w|<|S_v|$ then $w>v$, and if  $|S_w|=|S_v|$, then $\bar{w}>_{\lex}\bar{v}$ and again  $w>v$.

Case (d):  $u,v\in G(J)\setminus G(I\mm)$, $|S_u|=|S_v|$ and $\bar{u}=\bar{v}$ and $u>_{\lex}v$. There exist $l,r$ such that $r<l$, $t_r<h_r\leq a_r$ and $h_l<t_l\leq a_l$. We set  $w=(v/x_l)x_r$.  Then   $(w)\: v=(x_r)$ and $w\in G(J)$ with $w>v$, because   $|S_w|=|S_v|$, $\bar{w}=\bar{v}$ and $w>_{\lex}v$.

\medskip
In the next step we consider the general case where  $J=I_{(d;b_1,\ldots,b_n)}$ and $I\mm\subseteq J$. By the first step we can extend $I\mm$ to $L=I_{(d+1;a_1+1,\ldots,a_n+1)}$ by linear quotients. Since $I\mm\subseteq J$ it follows that  $a_i+1\leq b_i$ for  $i=1,\ldots,n$. Therefore, $L\subseteq J$, and hence it suffices to extend $L$ to $J$ by linear quotients.

Set $c_i=a_i+1$ for $i=1,\ldots,n$. It is enough to show that $L=I_{(d+1;c_1,\ldots,c_n)}$ can be extended to $K=I_{(d+1;c_1,\ldots, c_{s-1},c_s+1,c_{s+1},\ldots c_n)}$ for some $s\in[n]$. For monomials $u,v\in G(K)$, we say $u>v$, if $u\in G(L)$ and $v\in G(K)\setminus G(L)$ or $u,v\in G(K)\setminus G(L)$ and $u>_{\lex}v$.

We claim that with this order,  $L$ can be extended to $K$ by linear quotients.
We have  to show that for all $u=x_1^{h_1}\cdots x_n^{h_n}\in G(K)$ and all $v=x_1^{t_1}\cdots x_n^{t_n}\in G(K)\setminus G(L)$  with $u>v$  there exists $w\in G(K)$ with $w>v$  such that $(w)\:v=(x_j)$ and $x_j$ divides $u/\gcd(u,v)$.
 We distinguish two cases.

 (i)  $u\in G(L)$ and $v\in G(K)\setminus G(L)$. Since $v\in G(K)\setminus G(L)$, it follows that $t_s=c_s+1$, so $t_s>h_s$. On the other hand since $\deg(u)=\deg(v)$,  there exists $r\in[n]$ such that $h_r>t_r$. Let $w=(v/x_s)x_r$, then $(w)\:v=(x_r)$ and $w>v$ because $w\in G(L)$.

(ii)  $u,v\in G(K)\setminus G(L)$ and $u>_{\lex}v$. So there exist $l,r$ such that $r<l$, $t_r<h_r$ and $h_l<t_l$.
Let $w=(v/x_l)x_r$, then $(w)\:v=(x_r)$ and $w\in G(K)$ with $w>v$, because $w>_{\lex}v$.
\end{proof}

A monomial ideal $I$ is called {\em componentwise of  Veronese type}, if $I_{\langle j\rangle}$ is of Veronese type for all $j$.

\begin{Corollary}
\label{componentwiseveronese}
 Let $I$ be an ideal which is componentwise of Veronese type. Then $I$ has linear quotients.
\end{Corollary}

\begin{proof}
 It follows from Theorem \ref{extend}  that  $I_{\langle j\rangle}\mm$ can be extended to $I_{\langle j+1\rangle}$ by linear quotients for all $j$. Hence by \cite[Proposition 2.9]{SZ}  $I$ has linear quotients.
\end{proof}

\end{document}